\documentclass[10pt]{article}

\usepackage{amsmath,amsfonts,amssymb,amsthm}
\usepackage{color}
\usepackage{colortbl}
\usepackage{array}
\usepackage{mathrsfs}
\usepackage{dcolumn}
\usepackage{longtable}
\usepackage{hhline}
\usepackage{bm}

\newcommand{\bracketed}[2]{\genfrac{[}{]}{0pt}{0}{#1}{#2}}

\newtheorem{thm}{Theorem}[section]

\newtheorem{rmk}[thm]{Remark}
\newtheorem{cor}[thm]{Corollary}

\title{\bf Hankel Transform of the First Form $(q, r)$-Dowling Numbers}
\author{{\large\bf Roberto B. Corcino}\\
Research Institute for Computational Mathematics and Physics\\
Cebu Normal University\\
Cebu City, Philippines 6000}

\date{}

\begin{document}

\maketitle

\begin{abstract}
In this paper, the Hankel transform of the generalized $q$-exponential polynomial of the first form $(q,r)$-Whitney numbers of the second kind is established using the method of Cigler. Consequently, the Hankel transform of the first form $(q,r)$-Dowling numbers is obtained as special case. 

\end{abstract}

\section{Introduction}

The $r$-Dowling numbers $D_{m,r}(n)$ are defined in \cite{CHEON} as the sum of $r$-Whitney numbers of the second $W_{m,r}(n,k)$ \cite{corcino, Mezo2}. More precisely,   
$${D}_{m,r}(n):=\sum_{k=0}^nW_{m,r}(n,k).$$
These numbers are certain generalization of ordinary Bell numbers $B_n$ \cite{Com}, $r$-Bell numbers $B_{r}(n)$ \cite{Mezo3}, and noncentral Bell numbers $B_{n,a}$ \cite{Kout}. That is, when $m=1$, the $r$-Dowling numbers reduce to $r$-Bell numbers and noncentral Bell numbers. Furthermore, when $m=1, r=0$, these yield the ordinary Bell numbers.

\smallskip
J. Layman \cite{Lay} defined the Hankel transform of an integers sequence $(a_n)$ as a sequence of the following determinants $d_n$ of Hankel matrix of order $n$ 
\begin{equation}
d_n=\left|
\begin{matrix}
    a_0       & a_1       & a_2            & \dots & a_n \\
    a_1       & a_2 	& a_3            & \dots & a_{n+1} \\
    a_2       & a_3 	& a_4            & \dots & a_{n+2} \\
    \hdotsfor{5} \\
    a_n       & a_{n+1} & a_{n+2} & \dots & a_{2n}
\end{matrix}
\right|.
\end{equation}
Aigner \cite{Aig} derived the Hankel transform of the ordinary Bell numbers to be
\begin{equation}\label{det1}
det(B_{i+j})_{0\le i,j\le n}=\prod_{k=0}^nk!
\end{equation}
which is exactly the Hankel transform obtained by Mezo \cite{Mezo3} for $r$-Bell numbers using Layman's Theorem \cite{Lay} on the invariance of Hankel transform.

\smallskip
Using the method of Aiger \cite{Aig} and Layman's Theorem \cite{Lay},  the sequence of $(r,\beta)$-Bell numbers in \cite{Cor2, Cor3}, denoted by $\{G_{n,r,\beta}\}$, has been shown to possess the following Hankel transform (see \cite{Cor1}) 
$$H(G_{n,r,\beta})=\prod_{j=0}^n\beta^jj!.$$ 
It is worth mentioning that the $(r,\beta)$-Bell numbers are equivalent to the $r$-Dowling numbers $D_{m,r}(n)$, which are defined in \cite{CHEON} as
$$D_{m,r}(n)=\sum_{k=0}^nW_{m,r}(n,k)$$
with $W_{m,r}(n,k)$ denotes the $r$-Whitney numbers of the second kind introduced by Mezo in \cite{Mezo2}. In \cite{Cor1}, the authors have also tried to derive the Hankel transform of the sequence of $q$-analogue of $(r,\beta)$-Bell numbers. In this attempt, they used the $q$-analogue defined in \cite{Cor4}. But they failed to derive it.  

\smallskip
Just recently, another definition of $q$-analogue of $r$-Whitney numbers of the second $W_{m,r}[n,k]_q$ was introduced in \cite{Cor5, Cor6} by means of the following triangular recurrence relation
\begin{equation}\label{whitrec}
W_{m,r}[n,k]_{q}= q^{m(k-1)+r}W_{m,r}[n-1,k-1]_q+[mk+r]_{q} W_{m,r}[n-1,k]_{q}.
\end{equation}
From this definition, two more forms of the $q$-analogue were defined in \cite{Cor5, Cor6} as
\begin{align}
W^*_{m,r}[n,k]_q&:=q^{-kr-m\binom{k}{2}}W_{m,r}[n,k]_q\label{2ndform}\\
\widetilde{W}_{m,r}[n,k]_q&:=q^{kr}W^*_{m,r}[n,k]_q=q^{-m\binom{k}{2}}W_{m,r}[n,k]_q,\label{3rdform}
\end{align}
where $W^*_{m,r}[n,k]_q$ and $\widetilde{W}_{m,r}[n,k]_q$ denote the second and third forms of the $q$-analogue, respectively. Corresponding to these, three forms of $q$-analogues for $r$-Dowling numbers (or $(q,r)$-Dowling numbers) may be defined as follows:
\begin{align}
{D}_{m,r}[n]_q&:=\sum_{k=0}^nW_{m,r}[n,k]_q\label{1stformqDow}\\
{D}^*_{m,r}[n]_q&:=\sum_{k=0}^nW^*_{m,r}[n,k]_q\label{2ndformqDow}\\
\widetilde{D}_{m,r}[n]_q&:=\sum_{k=0}^n\widetilde{W}_{m,r}[n,k]_q.\label{3rdformqDow}
\end{align}
However, among the three forms of $(q,r)$-Dowling numbers, only the first form has not been given a Hankel transform. The third form was thoroughly studied in \cite{Cor6} and was given the Hankel transform as follows
\begin{equation}\label{HankelqRDN}
H(\widetilde{D}_{m,r}[n]_q)=q^{m\binom{n+1}{3}-rn(n+1)}[0]_{q^m}![1]_{q^m}!\ldots [n]_{q^m}![m]_q^{\binom{n+1}{2}}.
\end{equation}
This Hankel transform was derived using the Hankel transform of $q$-exponential polynomials in \cite{Ehr1}, the Layman's Theorem in \cite{Lay} and the Spivey-Steil Theorem in \cite{Spiv}. This method cannot be used to derive the Hankel transform of the first and second forms of $q$-analogues for $r$-Dowling numbers. But the method used by Cigler in \cite{Cig2} was found to be useful to derive the Hankel transforms for the second form of the $(q,r)$-Dowling numbers. The said Hankel transform was derived in \cite{Cor5}, which is given by
\begin{equation*}
H({D}^*_{m,r}[n]_q)=[m]_q^{\binom{n}{2}}q^{\binom{n}{3}+r\binom{n}{2}}\prod_{k=0}^{n-1}[k]_{q^m} !
\end{equation*}

\smallskip
Corcino et al. \cite{Cor7} have made a preliminary investigation for the first form $(q,r)$-Dowling numbers $D_{m,r}[n]_q$ by establishing an explicit formula expressed in terms of the first form $(q,r)$-Whitney numbers of the second kind and $(q,r)$-Whitney-Lah numbers. In this paper, the Hankel transform for the sequence $\left(D_{m,r}[n]_q\right)_{n=0}^{\infty}$ will be established using Cigler's method \cite{Cig2}. However, a more general form of $D_{m,r}[n]_q$, denoted by $\Phi_n[x,r,m]_q$, is considered, which is defined in polynomial form as follows:
\begin{equation}\label{genq1}
\Phi_{n}\left[ x,r,m\right]_q=\sum\limits_{k=0}^{n}W_{m,r}[n,k]_{q}x^{k}
\end{equation}
such that, when $x=1$, $\Phi_n[1,r,m]_q=D_{m,r}[n]_q$.

\section{Generalized $q$-Exponential Polynomials}
We may call $\Phi_{n}\left[ x,r,m\right]_q $ to be the generalized $q$-exponential polynomial of $q$-analogue of $r$-Whitney numbers of the second kind. Note that we can rewrite \eqref{genq1} as
\begin{align}
\Phi_{n-1}\left[ x,r,m\right]_q&=\sum\limits_{k=0}^{n-1}W_{m,r}[n-1,k]_{q}x^{k}\nonumber \\
\Phi_{n-1}\left[ qx,r,m\right]_q&=\sum\limits_{k=0}^{n-1}W_{m,r}[n-1,k]_{q}q^{rk+m\binom{k}{2}+k}x^{k}.\label{eqn_trans} 
\end{align}
The following theorem contains a recursive relation for $\Phi_{n}\left[ x,r,m\right]_q $.
\begin{thm}
The generalized $q$-exponential polynomials $\Phi_{n}\left[ x,r,m\right]_q $ of $q$-analogue of $r$-Whitney numbers of the second kind satisfy the following relation
\begin{align}
\Phi_n[x,r,m]_q &=\left[q^{r}x+(q^m-1)q^{r}x^2D_{q^m}+[r]_q+q^r[m]_qxD_{q^m}\right]\Phi_{n-1}[x,r,m]_q.\label{phi}
\end{align}
\end{thm}
\begin{proof}
Using \eqref{whitrec}, equation \eqref{genq1} can be written as
\begin{align*}
&\Phi_{n}\left[ x,r,m\right]_q =\sum\limits_{k=0}^{n}W_{m,r}[n,k]_{q}x^{k}\\
&\;\;\;\;=\sum\limits_{k=0}^{n}q^{mk-m+r}W_{m,r}[n-1,k-1]_qx^{k}+\sum\limits_{k=0}^{n}[mk+r]_qW_{m,r}[n-1,k]_qx^{k}\\
&\;\;\;\;=\sum\limits_{k=0}^{n-1}q^{m(k+1)-m+r}W_{m,r}[n-1,k]_qx^{k+1}+\left([r]_q+q^r[m]_qxD_{q^m}\right)\Phi_{n-1}[x,r,m]_q\\
&\;\;\;\;=x\sum\limits_{k=0}^{n-1}q^{mk+r}W_{m,r}[n-1,k]_qx^k+\left([r]_q+q^r[m]_qxD_{q^m}\right)\Phi_{n-1}[x,r,m]_q\\
&\;\;\;\;=xq^{r}\sum\limits_{k=0}^{n-1}q^{mk}W_{m,r}[n-1,k]_qx^k+\left([r]_q+q^r[m]_qxD_{q^m}\right)\Phi_{n-1}[x,r,m]_q,
\end{align*}
where $D_q$ denotes the $q$-derivative operator defined by
\begin{equation}\label{qDer}
D_qf(x)=\frac{f(x)-f(qx)}{(1-q)x}.
\end{equation}
Hence, using \eqref{eqn_trans}, we have
\begin{align}
\Phi_n[x,r,m]_q &=xq^{r}\Phi_{n-1}[q^mx,r,m]_q+\left([r]_q+q^r[m]_qxD_{q^m}\right)\Phi_{n-1}[x,r,m]_q.\label{genq2}
\end{align}
Note that \eqref{qDer} can be expressed as
\begin{align*}
f(qx)&=(q-1)xD_qf(x)+f(x)\\
f(q^mx)&=(q^m-1)xD_{q^m}f(x)+f(x)\\
f(q^mx)&=((q^m-1)xD_{q^m}+1)f(x).
\end{align*}
This implies that
$$\Phi_{n-1}[q^mx,r,m]_q=\left(1+(q^m-1)xD_{q^m}\right)\Phi_{n-1}[x,r,m]_q.$$
Thus, equation \eqref{genq2} can further be written as
\begin{align*}
\Phi_n[x,r,m]_q &=xq^{r}\left(1+(q^m-1)xD_{q^m}\right)\Phi_{n-1}[x,r,m]_q\\
&\;\;\;\;\;+\left([r]_q+q^r[m]_qxD_{q^m}\right)\Phi_{n-1}[x,r,m]_q\\
&=\left[q^{r}x+(q^m-1)q^{r}x^2D_{q^m}\right.\nonumber\\
&\left.\;\;\;\;\;+[r]_q+q^r[m]_qxD_{q^m}\right]\Phi_{n-1}[x,r,m]_q,
\end{align*}
which is exactly the desired relation.
\end{proof}
\begin{rmk}\rm
Let 
\begin{align*}
\hat{D}_{qx}&=\left[q^{r}x+(q^m-1)q^{r}x^2D_{q^m}+[r]_q+q^r[m]_qxD_{q^m}\right],\\
\tilde{D}_{qx}&=\left[q^{r}+(q^m-1)q^{r}xD_{q^m}+[r]_q+q^r[m]_qD_{q^m}\right],\\
\hat{D}_{qx}&=x\tilde{D}_{qx}.
\end{align*}
Then \eqref{phi} can be written as
\begin{align}
\Phi_{n}[x,r,m]_{q}&=\hat{D}_{qx}\Phi_{n-1}[x,r,m]_{q}.\label{Phi2}\\
\Phi_{n}[x,r,m]_{q}&=x\tilde{D}_{qx}\Phi_{n-1}[x,r,m]_{q}.\label{Phi3}
\end{align}
By repeated application of \eqref{Phi2},
\begin{align*}
\Phi_{n}[x,r,m]_{q}&=\hat{D}_{qx}\Phi_{n-1}[x,r,m]_{q}\\
&=\hat{D}_{qx}\left( \hat{D}_{qx}\Phi_{n-2}[x,r,m]_{q}\right)\\
&=\hat{D}_{qx}^{2}\Phi_{n-2}[x,r,m]_{q}\\
&\vdots\\
&=\hat{D}_{qx}^{n}\Phi_{0}[x,r,m]_{q}\\
&=\hat{D}_{qx}^{n}.
\end{align*}
\end{rmk}

\section{Hankel Transform of ${D}_{m,r}[n]_q$}
Let $\langle\langle x\rangle\rangle_{r,m,k}=\prod\limits_{j=0}^{k-1}\frac{\left(x-[r+jm]_q \right) }{q^{r+jm}}=q
^{-rk-m\binom{k}{2}}\langle x\rangle_{r,m,k}$.\\
The horizontal generating function of $W_{m,r}[n,k]_q$ is given by:
\begin{equation*}
\sum\limits_{k=0}^{n}W_{m,r}[n,k]_q\langle x\rangle_{r,m,k}=x^n
\end{equation*}
Define a linear functional $G_{r,q}$ by
\begin{equation*}
G_{r,q}\left( \langle\langle x\rangle\rangle_{r,m,n}\right)=a^n
\end{equation*}
and a linear operator $V_{r,q}$ by
\begin{equation*}
V_{r,q}\left( \langle\langle x\rangle\rangle_{r,m,n}\right)=x^n.
\end{equation*}
Then
\begin{align*}
V_{r,q}\left( x^n\right)  &=\sum\limits_{k=0}^{n}W_{m,r}[n,k]_q V_{r,q}\left(\langle x\rangle_{r,m,k}\right)\\
&=\sum\limits_{k=0}^{n}W_{m,r}[n,k]_q q^{rk+m\binom{k}{2}}V_{r,q}\left(\langle\langle x\rangle\rangle_{r,m,k}\right)\\
&=\sum\limits_{k=0}^{n}W_{m,r}[n,k]_q q^{rk+m\binom{k}{2}}x^k\\
&=\Phi_n[x,r,m]_q
\end{align*}
It can be easily verified from equation \eqref{Phi3} that
\begin{equation*}
V_{r,q}xV_{r,q}^{-1}=x\tilde{D}_{qx}. 
\end{equation*}
Consider the polynomial
\begin{equation*}
g_{n,q}(x,a,r,m)=\sum\limits_{k=0}^{n}(-a)^{k}q^{\binom{k}{2}}\bracketed{n}{k}_{q}\langle\langle x\rangle\rangle_{r,m,n-k}.
\end{equation*}
Then
\begin{align*}
V_{r,q}\left( g_{n,q}(x,a,r,m)\right)&=\sum\limits_{k=0}^{n}(-a)^{k}q^{\binom{k}{2}}\bracketed{n}{k}_{q}V_{r,q}\left( \langle\langle x\rangle\rangle_{r,m,n-k}\right) \\
&=\sum\limits_{k=0}^{n}(-a)^{k}q^{\binom{k}{2}}\bracketed{n}{k}_{q}x^{n-k}\\
&=p_{n,q}(x,a).
\end{align*}
This implies that $V_{r,q}^{-1}p_{n,q}(x,a)=g_{n,q}(x,a,r,m)$. Now,
\begin{align*}
V_{r,q}xg_{n,q}(x,a,r,m)&=V_{r,q}xV_{r,q}^{-1}p_{n,q}(x,a)\\
&=x\left[q^{r}+(q^m-1)q^{r}xD_{q^m}+[r]_q+q^r[m]_qD_{q^m}\right]p_{n,q}(x,a).
\end{align*}
Applying the operator to $p_{n,q}(x,a)$, we get
\begin{align*}
V_{r,q}xg_{n,q}(x,a,r,m)&=V_{r,q}xV_{r,q}^{-1}p_{n,q}(x,a)\\
&=x\left[q^{r}+(q^m-1)q^{r}xD_{q^m}+[r]_q+q^r[m]_qxD_{q^m}\right]p_{n,q}(x,a)\\
&=q^{r}xp_{n,q}(x,a)+(q^m-1)q^{r}x^2D_{q^m}p_{n,q}(x,a)+[r]_{q}p_{n,q}(x,a)\\
&\;\;\;\;+q^r[m]_qxD_{q^m}p_{n,q}(x,a).
\end{align*}
Note that
\begin{align*}
xp_{n,q}(x,a)&=\sum_{k=0}^{n}(-a)^kq^{\binom{k}{2}}\bracketed{n}{k}x^{n+1-k}\\
&=\sum_{k=0}^{n}(-a)^kq^{\binom{k}{2}}\left(\bracketed{n+1}{k}-q^{n+1-k}\bracketed{n}{k-1}\right)x^{n+1-k}\\
&=\sum_{k=0}^{n}(-a)^kq^{\binom{k}{2}}\bracketed{n+1}{k}x^{n+1-k}-\sum_{k=0}^{n}(-a)^kq^{\binom{k}{2}}q^{n+1-k}\bracketed{n}{k-1}x^{n+1-k}\\
&\;\;\;\;\;\;+(-a)^{n+1}q^{\binom{n+1}{2}}-(-a)^{n+1}q^{\binom{n+1}{2}}.
\end{align*}
\begin{align*}
xp_{n,q}(x,a)&=\sum_{k=0}^{n+1}(-a)^kq^{\binom{k}{2}}\bracketed{n+1}{k}x^{n+1-k}-\sum_{k=-1}^{n-1}(-a)^{k+1}q^{\binom{k+1}{2}}q^{n-k}\bracketed{n}{k}x^{n-k}\\
&\;\;\;\;\;\;-(-a)^{n+1}q^{\binom{n+1}{2}}\\
&=p_{n+1,q}(x,a)-(-a)q^n\sum_{k=0}^{n-1}(-a)^{k+1}q^{\binom{k+1}{2}-k}\bracketed{n}{k}x^{n-k}\\
&\;\;\;\;\;\;-(-a)^{n+1}q^{\binom{n+1}{2}}\\
&=p_{n+1,q}(x,a)+aq^n\sum_{k=0}^{n}(-a)^{k+1}q^{\binom{k}{2}}\bracketed{n}{k}x^{n-k}\\
&=p_{n+1,q}(x,a)+aq^np_{n,q}(x,a).
\end{align*}
So, $q^rxp_{n,q}(x,a)=q^rp_{n+1,q}(x,a)+aq^{n+r}p_{n,q}(x,a)$.  With $D_qp_{n,q}(x,a)=[n]_qp_{n-1,q}(x,a)$, we have
$$D_{q^m}p_{n,q}(x,a)=[n]_{q^m}p_{n-1,q}(x,a).$$
Hence,
$$(q^m-1)q^{r}x^2D_{q^m}p_{n,q}(x,a)=(q^m-1)q^{r}x^2[n]_{q^m}p_{n-1,q}(x,a)=(q^{mn}-1)q^{r}x^2p_{n-1,q}(x,a)$$
and
$$q^r[m]_qxD_{q^m}p_{n,q}(x,a)=q^r[m]_qx[n]_{q^m}p_{n-1,q}(x,a)$$
Thus, 
\begin{align*}
&V_{r,q}xg_{n,q}(x,a,r,m)=V_{r,q}xV_{r,q}^{-1}p_{n,q}(x,a)\\
&\;\;\;\;\;=q^rxp_{n,q}(x,a)+q^rp_{n+1,q}(x,a)+aq^{n+r}p_{n,q}(x,a)+[r]_{q}p_{n,q}(x,a)\\
&\;\;\;\;\;\;\;\;+(q^{mn}-1)q^{r}x^2p_{n-1,q}(x,a)+q^r[m]_qx[n]_{q^m}p_{n-1,q}(x,a)\\
&\;\;\;\;\;=q^r(p_{n+1,q}(x,a)+q^nap_{n,q}(x,a))+q^rp_{n+1,q}(x,a)+aq^{n+r}p_{n,q}(x,a)+[r]_{q}p_{n,q}(x,a)\\
&\;\;\;\;\;\;\;\;+(q^{mn}-1)q^{r}[p_{n+1,q}(x,a)+(q^na+q^{n-1}a)p_{n,q}(x,a)+q^{2n-2}a^2p_{n-1,a}(x,a)]\\
&\;\;\;\;\;\;\;\;+q^r[m]_q[n]_{q^m}(p_{n,q}(x,a)+q^{n-1}ap_{n-1,q}(x,a)\\
&\;\;\;\;\;=q^r(2+(q^{mn}-1))p_{n+1,q}(x,a)\\
&\;\;\;\;\;\;\;\;+(q^{n+r}a+aq^{n+r}+[r]_{q}+(q^{mn}-1)q^{r}(q^na+q^{n-1}a)+q^r[m]_q[n]_{q^m}) p_{n,q}(x,a)\\
&\;\;\;\;\;\;\;\;+((q^{mn}-1)q^{r}q^{2n-2}a^2+q^r[m]_q[n]_{q^m}q^{n-1}a)p_{n-1,q}(x,a)
\end{align*}
Applying the operator $V_{r,q}^{-1}:p_{n,q}(x,a)\mapsto g_{n,q}(x,a,r,m)$, then
\begin{align*}
&xg_{n,q}(x,a,r,m)=q^r(2+(q^{mn}-1))g_{n+1,q}(x,a,r,m)\\
&\;\;\;\;\;\;\;\;+(2q^{n+r}a+[r]_{q}+(q^{mn}-1)q^{r}(q^na+q^{n-1}a)+q^r[m]_q[n]_{q^m}) g_{n,q}(x,a,r,m)\\
&\;\;\;\;\;\;\;\;+((q^{mn}-1)q^{r}q^{2n-2}a^2+q^r[m]_q[n]_{q^m}q^{n-1}a)g_{n-1,q}(x,a,r,m)
\end{align*}
We set 
\begin{equation*}
h_{n,q}(x,a,r,m)=q^{m\binom{n}{2}+rn}g_{n,q}(x,a,r,m).
\end{equation*}
That is, 
\begin{equation*}
g_{n,q}(x,a,r,m)=q^{-m\binom{n}{2}-rn}h_{n,q}(x,a,r,m).
\end{equation*}
Then,
\begin{align*}
&xq^{-m\binom{n}{2}-rn}h_{n,q}(x,a,r,m)=q^r(2+(q^{mn}-1))q^{-m\binom{n+1}{2}-r(n+1)}h_{n+1,q}(x,a,r,m)\\
&\;\;\;\;+(2q^{n+r}a+[r]_{q}+(q^{mn}-1)q^{r}(q^na+q^{n-1}a)+q^r[m]_q[n]_{q^m}) q^{-m\binom{n}{2}-rn}h_{n,q}(x,a,r,m)\\
&\;\;\;\;\;+((q^{mn}-1)q^{r}q^{2n-2}a^2+q^r[m]_q[n]_{q^m}q^{n-1}a)q^{-m\binom{n-1}{2}-r(n-1)}h_{n-1,q}(x,a,r,m)
\end{align*}
\begin{align*}
&xh_{n,q}(x,a,r,m)=q^r(2+(q^{mn}-1))q^{-mn-r}h_{n+1,q}(x,a,r,m)\\
&\;\;\;\;+(2q^{n+r}a+[r]_{q}+(q^{mn}-1)q^{r}(q^na+q^{n-1}a)+q^r[m]_q[n]_{q^m}) h_{n,q}(x,a,r,m)\\
&\;\;\;\;\;+((q^{mn}-1)q^{r}q^{2n-2}a^2+q^r[m]_q[n]_{q^m}q^{n-1}a)q^{m(n-1)+r}h_{n-1,q}(x,a,r,m)
\end{align*}
\begin{align}
&xh_{n,q}(x,a,r,m)=(2+(q^{mn}-1))q^{-mn}h_{n+1,q}(x,a,r,m)\nonumber\\
&\;\;\;\;+(2q^{n+r}a+[r]_{q}+(q^{mn}-1)q^{r}(q^na+q^{n-1}a)+q^r[m]_q[n]_{q^m}) h_{n,q}(x,a,r,m)\nonumber\\
&\;\;\;\;\;+((q^{mn}-1)q^{r}q^{2n-2}a^2+q^r[m]_q[n]_{q^m}q^{n-1}a)q^{m(n-1)+r}h_{n-1,q}(x,a,r,m)\label{orthpol}
\end{align}
It is clear that 
\begin{align*}
G_{r,q}(h_{n,q}(x,a,r,m))&=G_{r,q}\left( q^{m\binom{n}{2}+rn}g_{n,q}(x,a,r,m)\right)\\
&=q^{m\binom{n}{2}+rn}\sum\limits_{k=0}^{n}(-a)^{k}q^{\binom{k}{2}}\bracketed{n}{k}_{q}G_{r,q}\left(\langle\langle x\rangle\rangle_{r,m,n-k} \right)\\ 
&=q^{m\binom{n}{2}+rn}\sum\limits_{k=0}^{n}(-a)^{k}q^{\binom{k}{2}}\bracketed{n}{k}_{q}a^{n-k}\\
&=q^{m\binom{n}{2}+rn}p_{n,q}(a,a)\\
&=0,
\end{align*}
\begin{equation*}
G_{r,q}\left( \langle\langle x\rangle\rangle_{rm,o}\right) =a^{0}=1
\end{equation*}
which implies
\begin{equation*}
G_{r,q}(1)=1.
\end{equation*}
and
\begin{align*}
g_{0,q}(x,a,r,m)&=\sum\limits_{k=0}^{0}(-a)^{k}q^{\binom{k}{2}}\bracketed{0}{k}_{q}\langle\langle x\rangle\rangle_{r,m,0-k}\\
&=(-a)^{0}q^{\binom{0}{2}}\bracketed{0}{0}_{q}\langle\langle x\rangle\rangle_{r,m,0}\\
&=1.
\end{align*}
It follows that
\begin{equation*}
h_{0,q}(x,a,r,m)=q^{m\binom{0}{2}+0}g_{0,q}(x,a,r,m)=1
\end{equation*}
and
\begin{equation*}
G_{r,q}\left( h_{0,q}(x,a,r,m)\right)=G_{r,q}(1)=1. 
\end{equation*}
Clearly, $G_{r,q}\left( [x]_{q}h_{n,q}(x,a,r,m)\right)=0 $ and from \eqref{orthpol},
\begin{equation*}
xh_{n,q}(x,a,r,m)=g(n)h_{n+1,q}(x,a,r,m)+f(n)h_{n,q}(x,a,r,m)+c(n)h_{n-1,q}(x,a,r,m)
\end{equation*}
where
\begin{align*}
g(n)&=(2+(q^{mn}-1))q^{-mn}\\
f(n)&=(2q^{n+r}a+[r]_{q}+(q^{mn}-1)q^{r}(q^na+q^{n-1}a)+q^r[m]_q[n]_{q^m})\\
c(n)&=((q^{mn}-1)q^{r}q^{2n-2}a^2+q^r[m]_q[n]_{q^m}q^{n-1}a)q^{m(n-1)+r}.
\end{align*}
Then,
\begin{align*}
x^{2}h_{n,q}(x,a,r,m)&=xxh_{n,q}(x,a,r,m)\\
&=x\left[ g(n)h_{n+1,q}(x,a,r,m)+f(n)h_{n,q}(x,a,r,m)+c(n)h_{n-1,q}(x,a,r,m)\right]\\
&=g(n)xh_{n+1,q}(x,a,r,m)+f(n)xh_{n,q}(x,a,r,m)\\
&\;\;\;\;\;+c(n)xh_{n-1,q}(x,a,r,m)\\
&=g(n)g(n+1)h_{n+2,q}(x,a,r,m)+g(n)f(n+1)h_{n+1,q}(x,a,r,m)\\
&\;\;\;\;\;+g(n)c(n+1)h_{n,q}(x,a,r,m)+g(n)f(n)h_{n+1,q}(x,a,r,m)\\
&\;\;\;\;\;+f^{2}(n)h_{n,q}(x,a,r,m)+f(n)c(n)h_{n-1,q}(x,a,r,m)\\
&\;\;\;\;\;+c(n)g(n-1)h_{n,q}(x,a,r,m)+c(n)f(n-1)h_{n-1,q}(x,a,r,m)\\
&\;\;\;\;\;+c(n)c(n-1)h_{n-2,q}(x,a,r,m)\\
&=g(n)g(n+1)h_{n+2,q}(x,a,r,m)\\
&\;\;\;\;\;+\left[ g(n)f(n+1)+g(n)f(n)\right] h_{n+1,q}(x,a,r,m)\\
&\;\;\;\;\;+\left[ g(n)c(n+1)+f^{2}(n)+c(n)g(n-1)\right] h_{n,q}(x,a,r,m)\\
&\;\;\;\;\;+\left[ f(n)c(n)+c(n)f(n-1)\right] h_{n-1,q}(x,a,r,m)\\
&\;\;\;\;\;+c(n)c(n-1)h_{n-2,q}(x,a,r,m).
\end{align*}
Applying the linear functional $G_{r,q}$ to $[x]_{q}^{2}h_{n,q}(x,a,r,m)$ gives,
\begin{align*}
G_{r,q}\left( x^{2}h_{n,q}(x,a,r,m)\right)&= 0\\
&\vdots\\
G_{r,q}\left( x^{k}h_{n,q}(x,a,r,m)\right)&= 0
\end{align*}
for $k<n$. For $k=n$,
\begin{align*}
x^{n}h_{n,q}(x,a,r,m)=&g(n)x^{n-1}h_{n+1,q}(x,a,r,m)+f(n)x^{n-1}h_{n,q}(x,a,r,m)\\
\;\;\;\;\;&+c(n)x^{n-1}h_{n-1,q}(x,a,r,m).
\end{align*}
Then,
\begin{align*}
G_{r,q}\left( x^{n}h_{n,q}(x,a,r,m)\right)=&g(n)G_{r,q}\left( x^{n-1}h_{n+1,q}(x,a,r,m)\right) +f(n)G_{r,q}\left( x^{n-1}h_{n,q}(x,a,r,m)\right) \\
\;\;\;\;\;&+c(n)G_{r,q}\left( x^{n-1}h_{n-1,q}(x,a,r,m)\right)\\
=&c(n)G_{r,q}\left( x^{n-1}h_{n-1,q}(x,a,r,m)\right)\\
=&c(n)c(n-1)G_{r,q}\left( x^{n-2}h_{n-2,q}(x,a,r,m)\right)\\
=&c(n)c(n-1)c(n-2)G_{r,q}\left( x^{n-3}h_{n-3,q}(x,a,r,m)\right)\\
&\vdots\\
=&c(n)c(n-1)c(n-2)\dots c(1)G_{r,q}\left( x^{0}h_{0,q}(x,a,r,m)\right)\\
=&\left[ \prod\limits_{i=1}^{n}c(i)\right] (1)\\
=&\prod\limits_{i=1}^{n}c(i)
\end{align*}
Since $x^{n}h_{n,q}(x,a,r,m)$ is a sequence of orthogonal polynomials with respect to linear functional $G_{r,q}$,
\begin{equation*}
d_{n,q}=G_{r,q}\left( x^{n}h_{n,q}(x,a,r,m)\right) =\prod\limits_{i=1}^{n}c(i)
\end{equation*}
where
\begin{equation*}
c(i)=((q^{mi}-1)q^{r}q^{2i-2}a^2+q^r[m]_q[i]_{q^m}q^{i-1}a)q^{m(i-1)+r}
\end{equation*}
Then
\begin{align*}
d_{n,q}(n,0)&=G_{r,q}\left[ [x]_{q}^{n}h_{n,q}(x,a,r,m)\right] \\
&=\prod\limits_{i=0}^{n-1}d_{i,q}\\
&=\prod\limits_{i=0}^{n-1}\left\lbrace \prod\limits_{j=1}^{i}\left[ ((q^{mj}-1)q^{r}q^{2j-2}a^2+q^r[m]_q[j]_{q^m}q^{j-1}a)q^{m(j-1)+r}\right]  \right\rbrace \\
&=\prod\limits_{i=0}^{n-1}\left\lbrace \prod\limits_{j=1}^{i}\left[ aq^{m(j-1)+2r}((q^{mj}-1)q^{2j-2}a+[m]_q[j]_{q^m}q^{j-1})\right]  \right\rbrace \\
&=\prod\limits_{i=0}^{n-1}q^{2ri}q^{m(1+2+3+\dots+(i-1))} a^i\prod\limits_{j=1}^{i}\left[ ((q^{mj}-1)q^{2j-2}a+[m]_q[j]_{q^m}q^{j-1})\right] \\
&=\prod\limits_{i=0}^{n-1}q^{2ri}q^{m\binom{i}{2}} a^i\prod\limits_{j=1}^{i}\left[ ((q^{mj}-1)q^{2j-2}a+[m]_q[j]_{q^m}q^{j-1})\right] \\
&=q^{2r(0+1+2+3+\dots+(n-1))+m\left[ \binom{2}{2}+\binom{3}{2}+\dots+\binom{n-1}{2}\right]}a^{0+1+2+\dots+(n-1)}\\
&\;\;\;\prod\limits_{i=0}^{n-1}\prod\limits_{j=1}^{i}\left[ ((q^{mj}-1)q^{2j-2}a+[m]_q[j]_{q^m}q^{j-1})\right] \\
&=q^{2r\binom{n}{2}+(m+1)\binom{n}{3}}a^{\binom{n}{2}}\prod\limits_{i=0}^{n-1}\prod\limits_{j=1}^{i}\left[ ((q^{mj}-1)q^{2j-2}a+[m]_q[j]_{q^m}q^{j-1})\right]\\
&=q^{2r\binom{n}{2}+m\binom{n}{3}}a^{\binom{n}{2}}\prod\limits_{i=0}^{n-1}q^{\binom{i}{2}}\prod\limits_{j=1}^{i}\left[ [mj]_q\left(1-q^j\left(\frac{1-q}{q}\right)a\right)\right]\\
&=q^{2r\binom{n}{2}+(m+1)\binom{n}{3}}a^{\binom{n}{2}}\prod\limits_{i=0}^{n-1}\prod\limits_{j=1}^{i}\left[ [mj]_q\left(1-q^{j-1}(1-q)a\right)\right].
\end{align*}
This result is stated formally in the following theorem.
\begin{thm}\label{thm1}
The Hankel transform of $\Phi_{n}[x,r,m]_{q}$ corresponding to the $0th$ Hankel determinant is given by
\begin{equation}\label{res1}
H\left( \Phi_{n}[x,r,m]_{q}\right) =q^{2r\binom{n}{2}+(m+1)\binom{n}{3}}a^{\binom{n}{2}}\prod\limits_{i=0}^{n-1}\prod\limits_{j=1}^{i}\left[ [mj]_q\left(1-q^{j-1}(1-q)a\right)\right].
\end{equation}
\end{thm}
Note that when $m=1$, \eqref{res1} yields
$$H\left( \Phi_{n}[x,r,1]_{q}\right) =q^{2r\binom{n}{2}+2\binom{n}{3}}a^{\binom{n}{2}}\prod\limits_{i=0}^{n-1}[i]_q!((1-q)a;q)_i$$
where 
\begin{equation*}
(x;q)_{i}=\prod\limits_{j=0}^{i-1}\left( 1-q^{j}x\right). 
\end{equation*}
This is exactly the result obtained by Cigler \cite{Cig2}.

\smallskip
As a direct consequence of Theorem \ref{thm1}, we have the following corollary, which contains the main result of this paper.
\begin{cor}
The Hankel transform of the sequence $\left(D_{m,r}[n]_q\right)_{n=0}^{\infty}$ is given by
\begin{equation*}
H\left(D_{m,r}[n]_q\right) =q^{2r\binom{n}{2}+(m+1)\binom{n}{3}}\prod\limits_{i=0}^{n-1}((1-q)a;q)_i\prod\limits_{j=1}^{i}[mj]_q.
\end{equation*}
\end{cor}
\begin{thm}
The Hankel transform of $\Phi_{n}[x,r,m]_{q}$ corresponding to the $1st$ Hankel determinant is given by
\begin{align*}
d_{n,q}(n,1)&=q^{2r\binom{n}{2}+(m+1)\binom{n}{3}}a^{\binom{n}{2}}\prod\limits_{i=0}^{n-1}((1-q)a;q)_i\prod\limits_{j=1}^{i}[mj]_q\\
&\;\;\;\;\;\;\;\sum\limits_{k=0}^{n}(-1)^{n}[x]_{q}^{k}q^{\binom{k}{2}}\bracketed{n}{k}_{q}\prod\limits_{j=0}^{k-1}\dfrac{[r+jm]_{q}}{q^{r+jm}}. 
\end{align*}
\end{thm}
\begin{proof}
From Gram-Schmidt orthogonalization process, we obtain
\begin{equation*}
d_{n,q}(n,1)=d_{n,q}(n,0)(-1)^{n}p_{n,q}(0)
\end{equation*}
where
$p_{n,q}(0)$ is a sequence of orthogonal polynomials i.e.,
\begin{equation*}
g_{n,q}(x,a,r,m)=\sum\limits_{k=0}^{n}\left( -a\right)^{k}q^{\binom{k}{2}}\bracketed{n}{k}_{q}\langle\langle x\rangle\rangle_{r,m,k}=p_{n,q}(x)
\end{equation*}
which implies
\begin{equation*}
p_{n,q}(0)=\sum\limits_{k=0}^{n}\left( -a\right)^{k}q^{\binom{k}{2}}\bracketed{n}{k}_{q}\langle\langle 0\rangle\rangle_{r,m,k}.
\end{equation*}
Since,
\begin{align*}
\langle\langle 0\rangle\rangle_{r,m,k}&=\prod\limits_{j=0}^{k-1}\dfrac{\left( [0]_{q}-[r+jm]_{q}\right) }{q^{r+jm}}\\
&=\prod\limits_{j=0}^{k-1}\dfrac{-[r+jm]_{q}}{q^{r+jm}}\\
&=\left( \dfrac{-[r]_{q}}{q^r}\right) \left(\dfrac{-[r+j]_q}{q^{r+m}} \right) \left( \dfrac{-[r+(k-1)m]_{q}}{q^{r+(k-1)m}}\right)\\
&=(-1)^{k}\prod\limits_{j=0}^{k-1}\dfrac{[r+jm]_{q}}{q^{r+jm}}.
\end{align*}
Then,
\begin{align*}
p_{n,q}(0)&=\sum\limits_{k=0}^{n}\left( -a\right)^{k}q^{\binom{k}{2}}\bracketed{n}{k}_{q}(-1)^{k}\prod\limits_{j=0}^{k-1}\dfrac{[r+jm]_{q}}{q^{r+jm}}\\
&=\sum\limits_{k=0}^{n}(-1)^{k}a^{k}q^{\binom{k}{2}}\bracketed{n}{k}_{q}(-1)^{k}\prod\limits_{j=0}^{k-1}\dfrac{[r+jm]_{q}}{q^{r+jm}}\\
&=\sum\limits_{k=0}^{n}a^{k}q^{\binom{k}{2}}\bracketed{n}{k}_{q}\prod\limits_{j=0}^{k-1}\dfrac{[r+jm]_{q}}{q^{r+jm}}
\end{align*}
which implies
\begin{equation*}
(-1)^{n}p_{n,q}(0)=\sum\limits_{k=0}^{n}(-1)^{n}[x]_{q}^{k}q^{\binom{k}{2}}\bracketed{n}{k}_{q}\prod\limits_{j=0}^{k-1}\dfrac{[r+jm]_{q}}{q^{r+jm}}.
\end{equation*}
Hence,
\begin{align*}
d_{n,q}(n,1)&=d_{n,q}(n,0)(-1)^{n}p_{n,q}(0)\\
&=q^{2r\binom{n}{2}+(m+1)\binom{n}{3}}a^{\binom{n}{2}}\prod\limits_{i=0}^{n-1}((1-q)a;q)_i\prod\limits_{j=1}^{i}[mj]_q\\
&\;\;\;\;\;\;\;\sum\limits_{k=0}^{n}(-1)^{n}[x]_{q}^{k}q^{\binom{k}{2}}\bracketed{n}{k}_{q}\prod\limits_{j=0}^{k-1}\dfrac{[r+jm]_{q}}{q^{r+jm}} 
\end{align*}
\end{proof}

\bigskip
\noindent{\bf Acknowledgement}. This research has been funded by Cebu Normal University (CNU) and the Commission  on Higher Education - Grants-in-Aid for Research (CHED-GIA).


\bigskip
\begin{thebibliography}{20}
\bibitem{Aig} M. Aigner, A Characterization of the Bell Numbers, {\it Discrete Math.} {\bf 205} (1999), 207-210.
\bibitem{Bro} A.Z. Broder, The r-Stirling Numbers, {\it Discrete Math.}  {\bf 49}(1984), 241-259.
\bibitem{Car1} Carlitz, L., $q$-Bernoulli numbers and polynomials. \textit{Duke Math. J.} \textbf{15} (1948) 987-1000.
\bibitem{CHAR} Ch.A. Charalambides and J. Singh, A review of the Stirling numbers, their generalization and statistical applications, {\it Commun. Statist.-Theory Meth.}{\bf 20}(8) (1988), 2533-2595. 
\bibitem{CHEON} G.S. Cheon and J.H. Jung, $r$-Whitney number of Dowling lattices, {\it Discrete Math.} {312}(2012), 2337--2348. 
\bibitem{Cig} J. Cigler,  A new $q$-Analog of Stirling numbers. \textit{Sitzunber. Abt. II.} \textbf{201.}(1992) 97-109.
\bibitem{Cig1} J. Cigler, Eine Charakterisierung der q-Exponentialpolynome, {\it Österreich. Akad. Wiss. Math.-Natur. Kl.
Sitzungsber. II}, {\bf 208} (1999) 143–157.
\bibitem{Cig2} J. Cigler, Hankel determinants of generalized q-exponential polynomials, 
 arXiv:0909.5581v1 [math.CO]. Available at https://arxiv.org/abs/0909.5581.
\bibitem{Com} L. Comtet, {\it Advanced Combinatorics}, Reidel, Dordrecht, The Netherlands, 1974.
\bibitem{CON} K. Conrad, A $q$-Analogue of Mahler Expansions I, {\it Adv. in Math.} {\bf 153} (2000),
185--230.
\bibitem{Cor0} C.B. Corcino, R.B. Corcino, J.M. Ontolan, C.M. Perez-Fernandez, and E.R. Cantallopez, The Hankel Transform of $q$-Noncentral Bell Numbers, {\it Int. J. Math. Math. Sci.}, Volume 2015, Article ID 417327, 10 pages.
\bibitem{Cor2} R.B. Corcino, The $(r,\beta)$-Stirling numbers. \textit{Mindanao Forum.} \textbf{14}(2) (1999)
\bibitem{Cor5} R.B. Corcino and J.T. Ca\~{n}ete, A $q$-Analogue of $r$-Whitney Numbers of the Second Kind,  arXiv:1907.03094v2 [math.CO]. Available at http://arxiv.org/abs/1907.03094v2.
\bibitem{Cor1} R.B. Corcino and C.B. Corcino, The Hankel Transform of Generalized Bell Numbers and Its q-Analogue, {\it Util. Math.}, {\bf 89} (2012), 297-309.
\bibitem{corcino} R. B. Corcino and C. B. Corcino, \textit{On the Maximum of the Generalized Stirling numbers}, Util. Math., {\bf 86} (2011), 241--256.
\bibitem{Cor3}  R.B. Corcino, C.B. Corcino, and  R. Aldema, Asymptotic Normality of the $(r,\beta)$-Stirling Numbers, {\it Ars Combin.}, {\bf 81} (2006), 81-96.
\bibitem{Cor6} R.B. Corcino, M.R. Latayada and M.P. Vega, Eur. J. Pure Appl. Math., {\bf 12}(2) (2019), 279--293.
\bibitem{Cor4} R.B. Corcino and C.B. Montero, A q-Analogue of Rucinski-Voigt Numbers, {\it ISRN Discrete Mathematics}, Volume {\bf 2012}, Article ID 592818, 18 pages, doi:10.5402/2012/592818.
\bibitem{Cor7} R.B. Corcino, J.M. Ontolan, M.R. Lobrigas, Eur. J. Pure Appl. Math., {\bf 14}(1) (2021), 65--81.
\bibitem{Cor5} R.B. Corcino, J.M. Ontolan, G.S. Rama, Eur. J. Pure Appl. Math., {\bf 12}(4) (2019), 1676--1688.
\bibitem{Cve} A. Cvetkovi$\acute{c}$, P. Rajkovi$\acute{c}$, and M. Ivkovi$\acute{c}$, Catalan numbers, The Hankel Transform and Fibonnaci numbers, {\it J. Integer Seq.}, {\bf 5}(2002), Article 02.1.3 
\bibitem{Des} M. Desainte-Catherine and X. G. Viennot, Enumeration of certain Young tableaux with bound height, {\it Combinatorie $\acute{E}$num$\acute{e}$rative (Montreal 1985)}, Lect. Notes in Math. {\bf 1234} (1986), 58-67.
\bibitem{Ehr1} R. Ehrenborg, Determinants of Involving $q$-Stirling Numbers, {\it Advances in Applied Mathematics}, {\bf 31}(2003), 630-642.
\bibitem{Ehr} R. Ehrenborg, The Hankel Determinant of exponential Polynomials, {\it Amer. Math. Monthly}, {\bf 107}(2000), 557-560
\bibitem{Gar} M. Garcia-Armas and B. A. Seturaman, A note on the Hankel transform of the central binomial coefficients, {\it J. Integer Seq.} {\bf 11}(2008), Article 08.5.8.
\bibitem{Gould} H.W. Gould, The $q$-Stirling Number of the First and Second Kinds. \textit{Duke Math. J.} \textbf{Vol. 28}(1968) 281-289.
\bibitem{KIM} M. S. Kim and J. W. Son, A Note on $q$-Difference Operators, {\it Commun. Korean Math. Soc.} {\bf 17} (2002), No. 3, pp. 423-430
\bibitem{Kout} M. Koutras. Non-Central Stirling Numbers and Some Applications. {\it Discrete Math.}{\bf 42} (1982):73-89.
\bibitem{Lay} J.W. Layman, The Hankel transform and some of its properties, {\it J. Integer Seq.} {\bf 4} (2001), Article 01.1.5.
\bibitem{MEDI} A. De Medicis and P. Leroux, Generalized Stirling Numbers, Convolution Formulae and $p,q$-Analogues, {\it Can. J. Math} {\bf 47}(3) (1995), 474-499.
\bibitem{Mezo1} I. Mez\H{o}, On the Maximum of $r$-Stirling Numbers, {\it Adv. in Appl. Math.} {\bf 41}(3) (2008), 293-306.
\bibitem{Mezo2} I. Mez\H{o}, A new formula for the Bernoulli polynomials, {\it Result. Math.} {\bf 58}(3) (2010), 329-335.
\bibitem{Mezo3} I. Mez\H{o}, The r-Bell numbers, {\it J. Integer Seq.} {\bf 14} (2011), Article 11.1.1.
\bibitem{Rad} C. Radoux, D$\acute{e}$terminat de Hankel construit sur des polynomes li$\acute{e}$s aux nombres de d$\acute{e}$rangements, {\it European Journal of Combinatorics} {\bf 12}(1991) 327-329
\bibitem{Rior} J. Riordan, Combinatorial identities, Wiley, New York, 1968
%\bibitem{RAIN} E.D. Rainville and P.E Bedient,{\it Elementary Differential Equation}, 7ed,Macmillan Publishing Company, Inc., 1989
\bibitem{Slo} N. J. Sloane, The On-line Encyclopedia of Integer Sequences, \textit{http://www.research.att.com/~njas/sequences}.
\bibitem{Spiv} M.Z. Spivey and L. L. Steil, The $k$-binomial transform and the Hankel transform, {\it J. Integer Sq.} {\bf 9}(2006), Article 06.1.1
\bibitem{Tamm} U. Tamm, Some aspects of Hankel matrices in coding theory and combinatorics, {\it Electron. J. combin.} {\bf 8}(1) A1(2001)
\bibitem{Vein} R. Vein and A. Dale, {\it Determinants and Their Applications in Mathematical Physics,} Springer, 1991
\bibitem{Zel} Daniel Zelinsky, {\it A First Course in Linear Algebra}, 2ed, Academic Press, Inc.,1973

\end{thebibliography}
\end{document}